\documentclass{amsart}

\title[The universal finite set]{The universal finite set}

\author{Joel David Hamkins}
 \address[Joel David Hamkins]
         {Professor of Logic, Oxford University and Sir Peter Strawson Fellow in Philosophy, University College, Oxford
         \fancyampersand\ Professor of Mathematics, Philosophy, and of Computer Science, The Graduate Center of The City University of New York \fancyampersand\ Professor of Mathematics, College of Staten Island of CUNY}
\email{jhamkins@gc.cuny.edu}
\urladdr{http://jdh.hamkins.org}

\author{W.~Hugh Woodin}
 \address[W.~Hugh Woodin]
        {Professor of Philosophy and of Mathematics, Department of Philosophy, Emerson Hall 209A, Harvard University,
25 Quincy Street Cambridge, MA 02138}
 \email{woodin@math.harvard.edu}

\thanks{Commentary can be made about this article on the first author's blog at \href{http://jdh.hamkins.org/the-universal-finite-set}{http://jdh.hamkins.org/the-universal-finite-set}.}

\usepackage{latexsym,amsfonts,amsmath,amssymb,mathrsfs}
\usepackage[hidelinks]{hyperref}
\usepackage{relsize}
\usepackage[rgb,dvipsnames]{xcolor}
\usepackage{tikz}
\usetikzlibrary{arrows,arrows.meta,petri,topaths,positioning,shapes,shapes.misc,patterns,calc,decorations.pathreplacing,hobby}
\usepackage{enumitem}

\newtheorem{theorem}{Theorem}

\newtheorem*{maintheorem*}{Main Theorem}
\newtheorem*{maintheorems*}{Main Theorems}
\newtheorem{corollary}[theorem]{Corollary}
\newtheorem*{corollary*}{Corollary}
\newtheorem*{corollaries*}{Corollaries}

\newtheorem{lemma}[theorem]{Lemma}

\newtheorem*{question*}{Question}

\newtheorem*{questions*}{Questions}
\newtheorem*{mainquestion*}{Main Question} 
\newtheorem*{openquestion*}{Open Question} 
\newtheorem{observation}[theorem]{Observation}

\newcommand{\QED}{\end{proof}}

\def\proclaim[#1]{{\bf #1}}
\def\BF#1.{{\bf #1.}}

\def\says#1:#2\par{\item[#1] #2\par}

\newcommand{\Los}{\L o\'s}

\newcommand{\Godel}{G\"odel}

\newcommand{\Lowe}{L\"owe}

\newcommand{\Levy}{L\'{e}vy}

\newcommand{\N}{{\mathbb N}}

\makeatletter
\newcommand{\dotminus}{\mathbin{\text{\@dotminus}}}
\newcommand{\@dotminus}{%
  \ooalign{\hidewidth\raise1ex\hbox{.}\hidewidth\cr$\m@th-$\cr}%
}
\makeatother

\newcommand{\of}{\subseteq}

\newcommand{\set}[1]{\{\,{#1}\,\}}

\newcommand{\satisfies}{\models}

\newcommand{\proves}{\vdash}

\DeclareMathOperator{\possible}{\text{\tikz[scale=.6ex/1cm,baseline=-.6ex,rotate=45,line width=.1ex]{\draw (-1,-1) rectangle (1,1);}}}
\DeclareMathOperator{\necessary}{\text{\tikz[scale=.6ex/1cm,baseline=-.6ex,line width=.1ex]{\draw (-1,-1) rectangle (1,1);}}}

\newcommand{\theoryf}[1]{{\rm #1}}

\newcommand{\Union}{\bigcup}

\newcommand{\smalllt}{\mathrel{\mathchoice{\raise2pt\hbox{$\scriptstyle<$}}{\raise1pt\hbox{$\scriptstyle<$}}{\raise0pt\hbox{$\scriptscriptstyle<$}}{\scriptscriptstyle<}}}
\newcommand{\smallleq}{\mathrel{\mathchoice{\raise2pt\hbox{$\scriptstyle\leq$}}{\raise1pt\hbox{$\scriptstyle\leq$}}{\raise1pt\hbox{$\scriptscriptstyle\leq$}}{\scriptscriptstyle\leq}}}

\newcommand{\boolval}[1]{\mathopen{\lbrack\!\lbrack}\,#1\,\mathclose{\rbrack\!\rbrack}}
\def\[#1]{\boolval{#1}}
\newbox\gnBoxA
\newdimen\gnCornerHgt
\setbox\gnBoxA=\hbox{\tiny$\ulcorner$}
\global\gnCornerHgt=\ht\gnBoxA
\newdimen\gnArgHgt
\def\gcode #1{%
\setbox\gnBoxA=\hbox{$#1$}%
\gnArgHgt=\ht\gnBoxA%
\ifnum     \gnArgHgt<\gnCornerHgt \gnArgHgt=0pt%
\else \advance \gnArgHgt by -\gnCornerHgt%
\fi \raise\gnArgHgt\hbox{\tiny$\ulcorner$} \box\gnBoxA %
\raise\gnArgHgt\hbox{\tiny$\urcorner$}}
\newcommand{\UnderTilde}[1]{{\setbox1=\hbox{$#1$}\baselineskip=0pt\vtop{\hbox{$#1$}\hbox to\wd1{\hfil$\sim$\hfil}}}{}}
\newcommand{\Undertilde}[1]{{\setbox1=\hbox{$#1$}\baselineskip=0pt\vtop{\hbox{$#1$}\hbox to\wd1{\hfil$\scriptstyle\sim$\hfil}}}{}}
\newcommand{\undertilde}[1]{{\setbox1=\hbox{$#1$}\baselineskip=0pt\vtop{\hbox{$#1$}\hbox to\wd1{\hfil$\scriptscriptstyle\sim$\hfil}}}{}}
\newcommand{\UnderdTilde}[1]{{\setbox1=\hbox{$#1$}\baselineskip=0pt\vtop{\hbox{$#1$}\hbox to\wd1{\hfil$\approx$\hfil}}}{}}
\newcommand{\Underdtilde}[1]{{\setbox1=\hbox{$#1$}\baselineskip=0pt\vtop{\hbox{$#1$}\hbox to\wd1{\hfil\scriptsize$\approx$\hfil}}}{}}

\newcommand{\fancyampersand}{{\usefont{OT1}{cmr}{m}{it} \&}}

\renewcommand{\iff}{\mathrel{\leftrightarrow}}

\def\<#1>{\left\langle#1\right\rangle}

\newcommand{\Ord}{\mathord{{\rm Ord}}}

\newcommand{\ZFC}{{\rm ZFC}}
\newcommand{\ZF}{{\rm ZF}}

\newcommand{\GCH}{{\rm GCH}}

\newcommand{\HOD}{{\rm HOD}}

\newcommand{\PA}{{\rm PA}}

\newcommand{\cell}[1]{\boxit{\hbox to 17pt{\strut\hfil$#1$\hfil}}}
\newcommand{\head}[2]{\lower2pt\vbox{\hbox{\strut\footnotesize\it\hskip3pt#2}\boxit{\cell#1}}}
\newcommand{\boxit}[1]{\setbox4=\hbox{\kern2pt#1\kern2pt}\hbox{\vrule\vbox{\hrule\kern2pt\box4\kern2pt\hrule}\vrule}}
\newcommand{\Col}[3]{\hbox{\vbox{\baselineskip=0pt\parskip=0pt\cell#1\cell#2\cell#3}}}
\newcommand{\tapenames}{\raise 5pt\vbox to .7in{\hbox to .8in{\it\hfill input: \strut}\vfill\hbox to
.8in{\it\hfill scratch: \strut}\vfill\hbox to .8in{\it\hfill output: \strut}}}
\newcommand{\Head}[4]{\lower2pt\vbox{\hbox to25pt{\strut\footnotesize\it\hfill#4\hfill}\boxit{\Col#1#2#3}}}
\newcommand{\Dots}{\raise 5pt\vbox to .7in{\hbox{\ $\cdots$\strut}\vfill\hbox{\ $\cdots$\strut}\vfill\hbox{\
$\cdots$\strut}}}
\usepackage[backend=bibtex,style=alphabetic,maxbibnames=15,maxcitenames=6,dateabbrev=false]{biblatex}
\renewcommand{\UrlFont}{\sffamily\smaller} 
\renewbibmacro{in:}{\ifentrytype{article}{}{\printtext{\bibstring{in}\intitlepunct}}} 
\DeclareFieldFormat{url}{{\relsize{-2}\UrlFont\url{#1}}} 
\DeclareFieldFormat{urldate}{
  (version \thefield{urlday}\addspace%
  \mkbibmonth{\thefield{urlmonth}}\addspace%
  \thefield{urlyear}\isdot)}
\DeclareFieldFormat{eprint:arxiv}{
  \ifhyperref
    {\href{http://arxiv.org/abs/#1}{%
        arXiv\addcolon\nolinkurl{#1}}\iffieldundef{eprintclass}{}{\UrlFont{\mkbibbrackets{\thefield{eprintclass}}}}}
    {arXiv\addcolon\nolinkurl{#1}\iffieldundef{eprintclass}{}{\UrlFont{\mkbibbrackets{\thefield{eprintclass}}}}}}
\bibliography{HamkinsBiblio,MathBiblio,WebPosts}
\newcommand\Val{\mathord{\rm Val}}

\newcommand{\setrandomcolor}{%
  \definecolor{randomcolor}{RGB}{\pdfuniformdeviate 256,\pdfuniformdeviate 256,\pdfuniformdeviate 256}%
}

\begin{document}

\begin{abstract}
We define a certain finite set in set theory $\set{x\mid\varphi(x)}$ and prove that it exhibits a universal extension property: it can be any desired particular finite set in the right set-theoretic universe and it can become successively any desired larger finite set in top-extensions of that universe. Specifically, \ZFC\ proves the set is finite; the definition $\varphi$ has complexity $\Sigma_2$, so that any affirmative instance of it $\varphi(x)$ is verified in any sufficiently large rank-initial segment of the universe $V_\theta$; the set is empty in any transitive model and others; and if $\varphi$ defines the set $y$ in some countable model $M$ of \ZFC\ and $y\of z$ for some finite set $z$ in $M$, then there is a top-extension of $M$ to a model $N$ in which $\varphi$ defines the new set $z$. Thus, the set shows that no model of set theory can realize a maximal $\Sigma_2$ theory with its natural number parameters, although this is possible without parameters. Using the universal finite set, we prove that the validities of top-extensional set-theoretic potentialism, the modal principles valid in the Kripke model of all countable models of set theory, each accessing its top-extensions, are precisely the assertions of \theoryf{S4}. Furthermore, if \ZFC\ is consistent, then there are models of \ZFC\ realizing the top-extensional maximality principle.
\end{abstract}

\maketitle

\section{Introduction}\label{Section.Introduction}

The second author~\cite{Woodin2011:A-potential-subtlety-concerning-the-distinction-between-determinism-and-nondeterminism} established the universal algorithm phenomenon, showing that there is a Turing machine program with a certain universal top-extension property in models of arithmetic. Namely, the program provably enumerates a finite set of natural numbers, but it is relatively consistent with \PA\ that it enumerates any particular desired finite set of numbers, and furthermore, if $M$ is any model of \PA\ in which the program enumerates the set $s$ and $t$ is any (possibly nonstandard) finite set in $M$ with $s\of t$, then there is a top-extension of $M$ to a model $N$ in which the program enumerates exactly the new set $t$. So it is a universal finite computably enumerable set, which can in principle be any desired finite set of natural numbers in the right arithmetic universe and become any desired larger finite set in a suitable larger arithmetic universe.\footnote{Woodin's theorem appears in~\cite{Woodin2011:A-potential-subtlety-concerning-the-distinction-between-determinism-and-nondeterminism}; Blanck and Enayat~\cite{BlanckEnayat2017:Marginalia-on-a-theorem-of-Woodin, Blanck2017:Dissertation:Contributions-to-the-metamathematics-of-arithmetic} removed the restriction to countable models and extended the result to weaker theories; Hamkins provided a simplified proof in~\cite{Hamkins:The-modal-logic-of-arithmetic-potentialism}; Shavrukov had reportedly circulated a similar argument privately, pointing out that it can be seen as an instance of the Berarducci~\cite{Berarducci1990:The-interpretability-logic-of-PA} construction.}

The first author~\cite{Hamkins.blog2017:The-universal-definition} inquired whether there is a set-theoretic analogue of this phenomenon, using $\Sigma_2$ definitions in set theory in place of computable enumerability. The idea was that just as a computably enumerable set is one whose elements are gradually revealed as the computation proceeds, a $\Sigma_2$-definable set in set theory is precisely one whose elements become verified at some level $V_\theta$ of the cumulative set-theoretic hierarchy as it grows. In this sense, $\Sigma_2$ definability in set theory is analogous to computable enumerability in arithmetic.

\begin{mainquestion*}[Hamkins]\label{Question.Universal-definition?}
 Is there a universal $\Sigma_2$ definition in set theory, one which can define any desired particular set in some model of \ZFC\ and always any desired further set in a suitable top-extension?
\end{mainquestion*}

The first author had noticed in~\cite{Hamkins.blog2017:The-universal-definition} that one can do this using a $\Pi_3$ definition, or with a $\Sigma_2$ definition, if one restricts to models of a certain theory, such as $V\neq\HOD$ or the eventual \GCH, or if one allows $\set{x\mid\varphi(x)}$ sometimes to be a proper class.

Here, we provide a fully general affirmative answer with the following theorem.

\goodbreak
\begin{maintheorem*}
There is a formula $\varphi(x)$ of complexity $\Sigma_2$ in the language of set theory, provided in the proof, with the following properties:
\begin{enumerate}
  \item \ZFC\ proves that $\set{x\mid \varphi(x)}$ is a finite set.
  \item In any transitive model of \ZFC\ and others, this set is empty.
  \item If $M$ is a countable model of \ZFC\ in which $\varphi$ defines the set $y$ and $z\in M$ is any finite set in $M$ with $y\of z$, then there is a top-extension of $M$ to a model $N$ in which $\varphi$ defines exactly $z$.
\end{enumerate}
\end{maintheorem*}

Notice that the main theorem provides a universal \emph{finite} set, rather than a universal set as in the main question, but we should like to emphasize that this actually makes for a stronger result, since the union of the universal finite set will be a universal set---the union of an arbitrary finite set, after all, is an arbitrary set. Similarly, by taking the union of just the countable members of the universal finite set, one achieves a universal countable set, and other kinds of universal sets are achieved in the same way for other cardinalities or requirements. The proof will show that one can equivalently formulate the main theorem in terms of finite sequences, rather than finite sets, so that the sequence is extended arbitrarily as desired in the top-extension. The sets $y$ and $z$ in statement (3) may be nonstandard finite, if $M$ is $\omega$-nonstandard.

In the final section, we shall apply our analysis to the theory of top-extensional set-theoretic potentialism. In particular, we shall show that the top-extensional potentialist validities of the countable models of set theory, for assertions with parameters, are precisely the modal assertions of \theoryf{S4}. We shall also provide models of the top-extensional maximality principle, models $M$ of \ZFC, for which whenever a sentence $\sigma$ is true in some top-extension of $M$ and all further top-extensions, then it is already true in $M$ and all its top-extensions.

\section{Background}\label{Section.Classical-background}

Let us briefly review some background classical results and constructions, which we shall make use of in the proof of the main theorem.

\subsection{Top-extensions}

A \emph{top-extension} of a model of set theory $\<M,\in^M>$, also known as a \emph{rank-extension}, is another model $\<N,\in^N>$ where $M$ is a submodel of $N$ and all the new sets in $N$ have rank above any ordinal of $M$. So every $V_\alpha^M$ is the same as $V_\alpha^N$, for ordinals $\alpha\in \Ord^M$. This is a \emph{topped}-extension, if there is a least new ordinal in $N\setminus M$.

One can similarly define a notion of top-extension for models of arithmetic, where the former model is an initial segment of the latter model, and in fact every model of \PA, whether countable or uncountable, has a nontrivial elementary top-extension. This theorem goes back in spirit perhaps to the 1934 definable ultrapower construction of Skolem~\cite{Skolem1934:Uber-die-Nicht-charakterisierbarkeit-der-Zahlenreihe} (reprinted in~\cite{Skolem1970:Selected-works-in-logic}), providing a top-extension of the standard model of arithmetic. Many years later, Mac~Dowell and Specker~\cite{MacDowellSpecker1961:Modelle-der-Arithmetik} proved the general case, handling arbitrary models of \PA, including uncountable models. Keisler and Morley~\cite{KeislerMorley1968:ElementaryExtensionsOfModelsOfSetTheory} extended the analysis to models of set theory, showing that every model of \ZF\ of countable cofinality (in particular, every countable model) has arbitrarily large elementary top-extensions.

\begin{theorem}[Keisler, Morley]\label{Theorem.Every-countable-model-has-elementary-top-extension}
  Every countable model of \ZFC\ has an elementary top-extension.
\end{theorem}

\begin{proof}[Proof sketch]
The theorem can be proved with the definable ultrapower method. After first forcing to add a global well-order, without adding sets, one then enumerates all the definable classes of ordinals in the model and all definable functions $f:\Ord\to M$, allowing the global well-order to appear in the definitions. In a construction with $\omega$-many steps, one constructs an $M$-ultrafilter $U$ on $\Ord^M$, which decides every definable class of ordinals in $M$, concentrates on final segments of $\Ord^M$, and has the property that every definable function $f:\Ord^M\to M$ that is bounded on a set in $U$ is constant on a set in $U$. One can simply build the ultrafilter step-by-step, using the pigeon-hole principle to make the bounded functions constant. Global choice is used to establish the \Los\ theorem for the definable ultrapower, which then provides the desired elementary top-extension.
\end{proof}

Keisler and Silver~\cite{KeislerSilver1971:End-extensions-models-of-set-theory} showed that the result is not true in general for uncountable models, for $\<V_\kappa,\in>$ has no elementary top-extension when $\kappa$ is the least inaccessible cardinal (and the weakly compact cardinals are characterized by similar extension properties, with predicates). Enayat~\cite[theorem 1.5(b)]{Enayat1984:OnCertainElementaryEndExtensionsOfModelsOfSetTheory} showed that indeed every consistent extension $T$ of \ZFC\ has a model of size $\aleph_1$ with no top-extension to a model of \ZFC\ (verified also for \ZF\ in unpublished work). That the rather classless models have no top-extensions is due to Kaufman and Enayat (Kaufman showed it for $\kappa$-like models, where $\kappa$ is regular; Enayat~\cite{Enayat1984:OnCertainElementaryEndExtensionsOfModelsOfSetTheory} proved the general case, with a different proof). Because uncountable models of set theory need not necessarily have top-extensions, one cannot expect a version of the main theorem for arbitrary uncountable models, even though the universal algorithm result holds for arbitrary models of arithmetic, including uncountable models.

In the countable case of theorem~\ref{Theorem.Every-countable-model-has-elementary-top-extension}, one cannot in general find a topped-extension, because the pointwise definable (see~\cite{HamkinsLinetskyReitz2013:PointwiseDefinableModelsOfSetTheory}) and more generally the Paris models can have no topped-extension, for such an extension would recognize the previous model as pointwise definable or a Paris model and therefore realize that it has only countably many countable ordinals (see~\cite[theorem 3.11]{Enayat2005:ModelsOfSetTheoryWithDefinableOrdinals}). For this reason, one cannot expect in the proof above to construct a \emph{normal} ultrafilter, since the resulting definable ultrapower would be topped.

\subsection{Locally verifiable properties are the $\Sigma_2$ properties}

The principal motivation for the use of $\Sigma_2$ in the main question lies in the characterization of the $\Sigma_2$ properties in set theory as those that are locally verifiable, in the sense that their truth can be verified by checking certain facts in only a bounded part of the set-theoretic universe, such as inside some rank-initial segment $V_\theta$ or inside the collection $H_\kappa$ of all sets of hereditary size less than $\kappa$. So let us review the basic facts.

\begin{lemma}[Folklore] For any assertion $\varphi$ in the language of set theory, the following are equivalent:
  \begin{enumerate}
    \item $\varphi$ is ZFC-provably equivalent to a $\Sigma_2$ assertion.
    \item $\varphi$ is ZFC-provably equivalent to an assertion of the form ``\,$\exists \theta\, V_\theta\models\psi$,'' where $\psi$ can be a statement of any complexity.
    \item $\varphi$ is ZFC-provably equivalent to an assertion of the form ``\,$\exists \kappa\, H_\kappa\models\psi$,'' where $\psi$ can be a statement of any complexity.
  \end{enumerate}
\end{lemma}

\begin{proof}
($3\to 2$) Since $H_\kappa$ is correctly computed inside $V_\theta$ for any $\theta>\kappa$, it follows that to assert that some $H_\kappa$ satisfies $\psi$ is the same as to assert that some $V_\theta$ thinks that there is some cardinal $\kappa$ such that $H_\kappa$ satisfies $\psi$.

($2\to 1$) The statement $\exists \theta\, V_\theta\models\psi$ is equivalent to the assertion that there is an ordinal $\theta$ and a sequence $v=\<v_\alpha\mid\alpha\leq\theta>$ of sets fulfilling the definition of the cumulative $V_\alpha$ hierarchy, namely, $v_0=\emptyset$, every $v_{\alpha+1}$ is the power set of $v_\alpha$ and $v_\lambda=\Union_{\alpha<\lambda}v_\alpha$ at limit ordinals $\lambda$, and finally such that $v_\theta\satisfies\psi$. This assertion has complexity $\Sigma_2$, since it begins with $\exists\theta\exists v$ and uses a universal quantifier to assert $s_{\alpha+1}=P(s_\alpha)$; the final clause $v_\theta\satisfies\psi$ has complexity $\Delta_0$, since all quantifiers of $\psi$ become bounded by $v_\theta$.

($1\to 3$) We make use of the \Levy\ absoluteness theorem, which asserts for any uncountable cardinal $\kappa$ that $H_\kappa\prec_{\Sigma_1} V$, a fact that can be proved by the Mostowski collapse of a witness in $V$ to find a suitable witness inside $H_\kappa$. For the theorem, consider any $\Sigma_2$ assertion $\exists x\,\forall y\, \varphi_0(x,y)$, where $\varphi_0$ has only bounded quantifiers. This assertion is equivalent to $\exists\kappa\, H_\kappa\models\exists x\,\forall y\,\varphi_0(x,y)$, where $\kappa$ is an uncountable cardinal, simply because if there is such a $\kappa$ with $H_\kappa$ having such an $x$, then by \Levy\ absoluteness, this $x$ works for all $y\in V$ since $H_\kappa\prec_{\Sigma_1}V$; and conversely, if there is an $x$ such that $\forall y\, \varphi_0(x,y)$, then this will remain true inside any $H_\kappa$ with $x\in H_\kappa$.
\end{proof}

Once a $\Sigma_2$ fact is witnessed in a model of set theory $M$, therefore, it follows that it remains true in all top-extensions of $M$. Thus, a $\Sigma_2$ definable set is enumerated as the cumulative $V_\theta$ hierarchy grows.

\subsection{Coding into the \GCH\ pattern}

Easton's theorem allows us to control precisely the pattern of success and failure of the generalized continuum hypothesis \GCH\ at the regular cardinals. For example, for any $a\of\omega$, we can find a forcing extension $V[G]$ where $n\in a$ if and only if the \GCH\ holds at $\aleph_n$. More generally, for any set of ordinals $A\of\gamma$, let us say that $A$ is coded into the \GCH\ pattern starting at $\delta$, if $2^\delta=\delta^{+(\gamma)}$ and then $A$ is coded into the \GCH\ pattern on the next $\gamma$ many successor cardinals, so that $\alpha\in A$ if and only if the \GCH\ holds at $\delta^{+(\gamma+\alpha+1)}$, for all $\alpha<\gamma$. And again, Easton's theorem shows how to code any desired set $A$ into the \GCH\ pattern of a forcing extension, starting at any desired regular cardinal $\delta$ and without adding bounded sets to $\delta$. Since every set is coded by a set of ordinals, this method provides a way to encode any desired set into the \GCH\ pattern of a forcing extension.

\subsection{\Godel-Carnap fixed point lemma}

Nearly all logicians are familiar with the \Godel\ fixed-point lemma, which asserts that for any formula $\varphi(x)$ in the language of arithmetic, there is a sentence $\psi$ for which $\PA\proves\psi\iff\varphi(\gcode{\psi})$. One may easily adapt the lemma to the language of set theory, using a weak fragment of set theory in place of \PA. The usual proof also works when $\varphi(x,\vec y)$ has additional free variables, in which case the fixed point is a formula $\psi(\vec y)$ with $\forall \vec y[\psi(\vec y)\iff\varphi(\gcode{\psi},\vec y)]$, and this formulation is known as the \Godel-Carnap fixed-point lemma. One can use this latter form, as we shall in the proof of the main theorem, to define a function $n\mapsto(\beta,k,y)$ by specifying its graph $\psi(n,\beta,k,y)$, which holds when $\varphi(\gcode{\psi},n,\beta,k,y)$. That is, one defines the function by a property that makes reference to the defining formula itself. The \Godel-Carnap fixed-point lemma ensures that indeed there is such a formula $\psi$.

\section{Special cases of the main theorem}\label{Section.Special-cases}

We shall approach the main theorem by first proving some special cases of it, specifically by providing separate definitions of the universal finite set for the cases of countable $\omega$-nonstandard models and countable $\omega$-standard models. Later, in section~\ref{Section.Proof-of-main-theorem}, we shall explain how to merge these two definitions into a single universal definition that will establish the main theorem by working with all countable models simultaneously.

\subsection{Process A}

Consider the following process, proceeding in finite stages, which will establish the main theorem for the case of $\omega$-nonstandard models of set theory. At stage $n$, if the earlier stages were successful, we search for an ordinal $\beta_n$, a natural number $k_n$ and a finite set $y_n$ fulfilling a certain locally verifiable property, and if we find them, we shall say that stage $n$ is successful and declare $\varphi_A(a)$ to hold of every $a\in y_n$. The ordinals $\beta_n$ will be increasing
 $$\beta_0<\beta_1<\cdots<\beta_n,$$
while the natural numbers, crucially, will be decreasing
 $$k_0>k_1>\cdots>k_n.$$
It follows that there can be only finitely many successful stages, and consequently the set $\set{a \mid\varphi_A(a)}$ ultimately defined by $\varphi_A$ will be finite. The process itself amounts to the map $n\mapsto(\beta_n,k_n,y_n)$, and our definition of this map will make explicit reference to the outcome of the map defined in the same way in various other models of set theory. This may seem initially to be a circular definition, but the \Godel-Carnap fixed-point lemma, explained in section~\ref{Section.Classical-background}, shows that indeed there is a definition solving this recursion; the situation is analogous to typical uses of the Kleene recursion theorem in the constructions of computability theory. Fix a computable enumeration of the theory \ZFC, and let $\ZFC_k$ denote the theory resulting from the first $k$ axioms in this enumeration.

To begin, stage $0$ is successful if there is a pair $\<\beta_0,k_0>$, minimal respect to the lexical order, such that $\beta_0$ is a $\beth$-fixed point and the structure $\<V_{\beta_0},\in>$ has no topped-extension to a model $\<N,\in^N>$ of $\ZFC_{k_0}$ satisfying the assertion, ``stage $0$ succeeds with $\beta_0$ and is the last successful stage, and this remains true after any $P(\beta_0)$-preserving forcing.'' In this case, let $y_0$ be the finite set coded into the \GCH\ pattern starting at $\beta_0^+$, if indeed there is a finite set coded there (take $y_0=\emptyset$ otherwise), and we declare $\varphi_A(a)$ to be true for all $a\in y_0$.

More generally, stage $n$ is successful, if all the previous stages were successful and we find a $\beth$-fixed point $\beta_n$ above all previous $\beta_i$ and a natural number $k_n$ less than all previous $k_i$ for $i<n$, taking $\<\beta_n,k_n>$ to be minimal in the lexical order, such that the structure $\<V_{\beta_n},\in>$ has no topped-extension to a model $\<N,\in^N>$ of $\ZFC_{k_n}$ satisfying the assertion, ``stage $n$ succeeds with $\beta_n$ and is the last successful stage, and this remains true after any $P(\beta_n)$-preserving forcing.'' In this case, let $y_n$ be the finite set coded into the \GCH\ pattern starting at $\beta_n^+$, using $y_n=\emptyset$ if there is no finite set coded there, and declare $\varphi_A(a)$ to hold for all $a\in y_n$.

Let us verify several facts about this construction. First, at each stage the question of whether or not there is such a model $\<N,\in^N>$ as described in the process is verifiable in $V_{\beta_n+1}$ and the encoded set $y_n$ is revealed in any sufficiently large $V_\theta$. So the map $n\mapsto (\beta_n,k_n,y_n)$ has complexity $\Delta_2$ and $\varphi_A$ therefore has complexity $\Sigma_2$. Because the $k_n$'s are descending in the natural numbers, there can be only finitely many successful stages, and so $\set{a\mid\varphi_A(a)}$ is finite.

We claim that if $M$ is a model of \ZFC\ in which stage $n$ is successful, then the number $k_n$ must be nonstandard and in particular, $M$ must be an $\omega$-nonstandard model. Since the $k$-sequence is descending, we might as well consider only the last successful stage in $M$. Note that if stage $n$ is successful, then this remains true after any $P(\beta_n)$-preserving forcing, with the same $\beta_n$ and $k_n$, since such forcing cannot create fundamentally new models $N$ of size $\beta_n$, and so by moving to the relevant successive forcing extensions and increasing $n$ if necessary, we may also assume without loss that $n$ is the last successful stage in $M$ and remains the last successful stage in any $P(\beta_n)$-preserving forcing extension of $M$. For stage $n$ to have been successful, it means by definition that $M$ thinks there is no topped-extension of $\<V_{\beta_n},\in>$ to a model $\<N,\in^N>$ satisfying $\ZFC_{k_n}$ in which stage $n$ succeeds with $\beta_n$, is the last successful stage, and this remains true after all $P(\beta_n)$-preserving forcing. But since $n$ actually did succeed with $\beta_n$, was the last successful stage in $M$ and this was preserved by this kind of forcing, however, it follows by the reflection theorem that for any standard $k$ there are many ordinals $\theta>\beta_n$ for which $V_\theta$ satisfies $\ZFC_k$ and agrees that $n$ was the last successful stage, that it succeeded with $\beta_n$ and that this is preserved by all $P(\beta_n)$-preserving forcing. Since $V_\theta$ is a topped-extension of $V_{\beta_n}$, this means $k<k_n$ for all such $k$ and so $k_n$ must be nonstandard. In particular, process $A$ has no successful stages at all in an $\omega$-standard model of \ZFC.

Let us now verify the universal top-extension property. Suppose that $M$ is a countable $\omega$-nonstandard model of \ZFC\ in which $\varphi_A$ happens to define the set $y$ and that $z$ is a (possibly nonstandard) finite set in $M$ with $y\of z$. Let $n$ be the first unsuccessful stage in $M$, and let $M^+$ be any countable elementary top-extension of $M$ (not necessarily topped); such a model exists by the arguments of section~\ref{Section.Classical-background}. Let $\beta$ be any $\beth$-fixed point of $M^+$ above $M$ and let $k$ be any nonstandard natural number of $M^+$ below all $k_i$ for $i<n$. Since stage $n$ was the first unsuccessful stage in $M$, it follows by elementarity that this is also true in $M^+$. In particular, stage $n$ did not succeed with $\beta$ and $k$, even though $\beta>\beta_i$ and $k<k_i$ for $i<n$. In order for this stage to have been unsuccessful---and this is the key step of the argument, explaining why the definition is the way that it is---it must have been that there is a topped-extension of $\langle V_\beta^{M^+},\in\rangle$ to a model $\<N,\in^N>$ satisfying $\ZFC_k$ and the assertion, ``stage $n$ is successful with $\beta$ and is the last successful stage, and this remains true after any $P(\beta)$-preserving forcing,'' for otherwise the lack of such a model $N$ would have caused stage $n$ to be successful in $M^+$. So $\beta=\beta_n^N$. Since $k$ is nonstandard, the model $\<N,\in^N>$ is a model of the actual \ZFC, and since it top-extends $V_{\beta}^{M^+}$, it also top-extends the original model $M$. So we have found a top-extension of $M$ in which stage $n$ is the last successful stage and this is preserved by $P(\beta)$-preserving forcing. Let $N[G]$ be a forcing extension of $N$ in which the set $z$ is coded into the \GCH\ pattern starting at $\beta^+$. (Since $N$ is countable, we may easily find an $N$-generic filter $G$, and the Boolean ultrapower construction leads to the forcing extension $N[G]$ with ground model $N$, even when $N$ is nonstandard; see~\cite{HamkinsSeabold:BooleanUltrapowers} for further explanation of forcing over ill-founded models.) Since this forcing preserves $P(\beta)$, it follows that $N[G]$ thinks stage $n$ is the last successful stage. And since we've now coded $z$ into the \GCH\ pattern starting at $\beta^+$, and $y\of z$, it follows that $N[G]$ thinks $\set{a\mid \varphi_A(a)}$ is precisely $z$. So we have achieved the universal top-extension property, as desired, for the case of $\omega$-nonstandard models.

\subsection{Process B}

Let us now describe an alternative process, which will work with the $\omega$-standard models (we have already pointed out that process $A$ has no successful stages in such models).

Stage $0$ of process $B$ is successful, if there is a $\beth$-fixed point $\gamma_0$ such that after forcing to collapse $\gamma_0$ to $\omega$, the now-countable structure $\<V_{\gamma_0},\in>$ has no topped-extension to a model $\<N,\in^N>$ satisfying \ZFC\ (not a fragment) and the assertion, ``stage $0$ is successful with $\gamma_0$ and is the last successful stage, and this remains true after any $P(\gamma_0)$-preserving forcing.'' This is a true $\Pi^1_1$ assertion in the collapse extension $V[g_0]$, which is therefore equivalent to the well-foundedness of the canonically associated tree, which has some rank $\lambda_0$ in $V[g_0]$; by homogeneity this does not depend on the generic filter. In this case, taking the pair $\<\gamma_0,\lambda_0>$ to be lexically least, we let $y_0$ be the finite set coded into the \GCH\ pattern starting at $\gamma_0^+$, or $\emptyset$ if there is not a finite set coded there, and declare that $\varphi_B(a)$ holds for all $a\in y_0$.

More generally, stage $n$ of process $B$ is successful, if there is a $\beth$-fixed point $\gamma_n$ larger than all previous $\gamma_i$ and an ordinal $\lambda_n$ strictly smaller than all previous $\lambda_i$ for $i<n$, such that in the collapse extension of $\gamma_n$ to $\omega$, the now-countable structure $\<V_{\gamma_n},\in>$ has no topped-extension to a model $\<N,\in^N>$ satisfying (full) \ZFC\ and the assertion, ``stage $n$ succeeds with $\gamma_n$ and is the last successful stage, and this remains true after any $P(\gamma_n)$-preserving forcing,'' such that furthermore the rank of the canonical well-founded tree witnessing this as a true $\Pi^1_1$ assertion is precisely $\lambda_n$. In this case, taking $\<\gamma_n,\lambda_n>$ to be lexically least, we interpret the \GCH\ pattern starting at $\gamma_n^+$ to code a finite set $y_n$, and we declare $\varphi_B(a)$ to hold for all $a\in y_n$.

This completes the description of process $B$. Let us prove that it has the desired properties. The map $n\mapsto (\gamma_n,\lambda_n,y_n)$ has a graph with complexity $\Delta_2$, since any sufficiently large $V_\theta$ can verify whether or not $(\gamma,\lambda,y)$ are as desired, and so the definition $\varphi_B(a)$ has complexity $\Sigma_2$. Since the $\lambda_n$-sequence is strictly descending, there will be only finitely many successful stages and therefore $\set{a\mid\varphi_B(a)}$ will be finite.

If stage $n$ is successful in an $\omega$-standard model $M$, then we claim that the associated ordinal $\lambda_n$ must come from the ill-founded part of the model, and in particular, $M$ must not be well-founded. To see this, it suffices to consider the case that $n$ is the last successful stage in $M$ and furthermore, by moving to finitely many successive forcing extensions, if necessary, that this is preserved by any $P(\gamma_n)$-preserving forcing. Thus, $M$ itself is a topped-extension of $V_{\gamma_n}^M$ satisfying \ZFC\ and the assertion, ``stage $n$ succeeds with $\gamma_n$ and is the last successful stage, and this remains true after any $P(\gamma_n)$-preserving forcing.'' So the tree whose well-foundedness is equivalent to the non-existence of such a model cannot be actually well-founded, and so the rank $\lambda_n$ of that tree must be in the ill-founded part of $M$. In particular, in any transitive model of \ZFC, there are no successful stages at all in process $B$.

Let us now verify the universal top-extension property for countable $\omega$-standard models. Suppose that $M$ is a countable $\omega$-standard model of \ZFC\ in which $\varphi_B$ happens to define the set $y$ and that $z$ is a finite set in $M$ with $y\of z$. Let $n$ be the first unsuccessful stage in $M$, and let $M^+$ be any countable elementary top-extension of $M$, not necessarily topped. Let $\gamma$ be any $\beth$-fixed point of $M^+$ above $M$. We claim that there is a topped-extension of the structure $\langle V_\gamma^{M^+},\in\rangle$ to a countable model $\<N,\in^N>$ of \ZFC\ satisfying the assertion, ``stage $n$ succeeds with $\gamma$ and is the last successful stage, and this remains true after any $P(\gamma)$-preserving forcing.'' Suppose toward contradiction that this is not true. Then in particular, in the forcing extension $M^+[g]$ collapsing $\gamma$ to $\omega$, there is no such model. Let $T$ be the tree whose well-foundedness is canonically equivalent to the nonexistence of such a model, and let $\lambda$ be the rank of this tree. Since $M$ is an $\omega$-model, this is the same tree as considered in $M^+[g]$ during process $B$. Since $T$ is actually well-founded in $V$, as the statement that there is no such model $N$ was true in $V$ by assumption, it follows that $T$ is also well-founded in $M^+[g]$, and furthermore, the ranks must agree. So $\lambda$ is an ordinal in the well-founded part of $M$, and consequently $\lambda<\lambda_i$ for all $i<n$, since we have previously noted that those ordinals, if defined, reside in the ill-founded part of $M$. But this situation shows that stage $n$ must be successful in $M^+$, contrary to the choice of $n$ as the first unsuccessful stage there. So our initial assumption must have been false, and therefore indeed, there must be such a countable model $N$ that is a topped-extension of $V_\gamma^{M^+}$ in which stage $n$ succeeds at $\gamma$ and is the last successful stage, and this remains true in all $P(\gamma)$-preserving forcing extensions. Let $N[G]$ be a forcing extension of $N$ that codes the set $z$ into the \GCH\ pattern starting at $\gamma_{n}^+$. Since the earlier stages succeed at stages in $M$, far below $\gamma$, the only new sets fulfilling $\varphi_B$ in $N[G]$ will come from stage $n$, and this will include exactly the elements of $z$. So in $N[G]$, the set $\set{a\mid\varphi_B(a)}$ is precisely the set $z$, as desired for the universal top-extension property.

\section{Proof of the main theorem}\label{Section.Proof-of-main-theorem}

We shall now merge the two processes of section~\ref{Section.Special-cases} into a single process that provides a $\Sigma_2$ definition with the universal top-extension property stated in the main theorem.

\begin{proof}[Proof of the main theorem] We describe process $C$, which simply runs processes $A$ and $B$ concurrently, except that the results of process $B$ are accepted only if no stage of $A$ has yet been successful, that is, as the rank hierarchy $V_\theta$ grows. In other words, once process $A$ is successful, then process $C$ proceeds further only with process $A$. Let $\varphi(a)$ hold if $a$ is accepted by process $C$.

It follows that any instance of $\varphi(a)$ is verified in some $V_\theta$ and so $\varphi$ has complexity $\Sigma_2$. Since processes $A$ and $B$ accept only finitely many objects, it follows that $\set{a\mid\varphi(a)}$ will also be a finite set. We have mentioned that neither process $A$ nor $B$ accepts any sets in a transitive model of \ZFC, and so the set $\set{a\mid\varphi(a)}$ is empty in any such model.

Finally, we verify the top-extension property. Suppose that $M$ is a countable model of \ZFC\ in which process $\varphi$ happens to define the finite set $y$, and that $z$ is a finite set in $M$ with $y\of z$. If $M$ is $\omega$-standard, then process $A$ will have no successful stages, and so process $C$ will amount in this case to process $B$, which we have already proved has the top-extension property. So there is a top-extension in which $\varphi$ defines exactly the set $z$.

If $M$ is $\omega$-nonstandard and process $A$ has had at least one successful stage in $M$, then process $C$ will continue only with process $A$, which has the top-extension property for $\omega$-nonstandard models. So we can find a top-extension $N$ in which process $A$ and hence also $C$ accepts exactly the elements of $z$, as desired.

The remaining case occurs when $M$ is $\omega$-nonstandard, but process $A$ has not yet had a successful stage $0$. Perhaps process $B$ has been successful or perhaps not. Let $M^+$ be an elementary top-extension of $M$, not necessarily topped, and let $\beta$ be any $\beth$-fixed point of $M^+$ above all the ordinals of $M$ and let $k$ be any nonstandard natural number of $M$. By elementarity, $\varphi$ defines the set $y$ in $M^+$ and process $A$ has not had a successful stage $0$ in $M^+$. In particular, this stage did not succeed with $\<\beta,k>$, and therefore there must be a topped-extension of $\langle V_\beta^{M^+},\in\rangle$ to a model $\<N,\in^N>$ in $M^+$ that satisfies $\ZFC_k$ and thinks stage $0$ is successful in process $A$ with $\beta$ and is the last successful stage, and this is preserved by further $P(\beta)$-preserving forcing. Let $N[G]$ be a forcing extension of $N$ in which $z$ is coded into the \GCH\ pattern starting at $\beta^+$. This forcing is $P(\beta)$-preserving, and so stage $0$ of process $A$ is successful at $\beta$ in $N[G]$ and is the last successful stage. Note that process $B$ has had no additional successful stages below $\beta$ in $N[G]$, except those that were already successful in $M$, since $N[G]$ agrees with $M^+$ up to $V_\beta$. Thus, process $C$ does nothing new in $N[G]$ except the now successful stage $0$ of process $A$ using $\beta$, and since $z$ is the set coded into the \GCH\ pattern at $\beta^+$, it follows that $\varphi$ will define precisely the set $z$ in $N[G]$, thereby providing the desired top-extension of $M$.
\end{proof}

Our use of \GCH\ coding in the proof of the main theorem can be replaced with essentially any of the other standard coding methods, such as $\Diamond^*_\kappa$ coding~\cite{Brooke-Taylor2009:LargeCardinalsAndDefinableWellOrders}. All that is needed about the coding is that one can force to code any given set starting at any given cardinal, while not adding bounded sets to that cardinal. By using the alternative coding methods, one could for example arrange to have \GCH\ in the extension models, provided the original model had \GCH.

\section{Maximal $\Sigma_2$ theories}

The existence of the universal finite set provided by the main theorem has a bearing on the question of maximal $\Sigma_2$ theories in models of set theory. A model of set theory $M\satisfies\ZFC$ has a \emph{maximal} $\Sigma_2$ theory, if the $\Sigma_2$ fragment of the theory of $M$ is a maximal consistent $\Sigma_2$ extension of \ZFC. In other words, any $\Sigma_2$ assertion $\sigma$ that is consistent with $\ZFC$ plus the $\Sigma_2$ theory of $M$ is already true in $M$.

\begin{observation}\label{Observation.Maximal-Sigma_2-theory}
 If there is a model of \ZFC, then there is a model of \ZFC\ with a maximal $\Sigma_2$ theory.
\end{observation}

\begin{proof}
Suppose that \ZFC\ is consistent. Enumerate the $\Sigma_2$ sentences $\sigma_0$, $\sigma_1$, $\sigma_2$, and so on. Let us form a certain theory $T$. We start with \ZFC, which we have already assumed is consistent. At stage $n$, we add the sentence $\sigma_n$, if the resulting theory remains consistent. Let $T$ be the theory after all stages are completed. (This theory is just like the usual construction to complete a theory, except that we consider only $\Sigma_2$ sentences and we have a preference for adding the $\Sigma_2$ sentences, rather than their negations.) The theory $T$ is consistent, by compactness, since it is consistent at each stage. If $M$ is a model of $T$, then $M$ is a model of \ZFC\ with a maximal $\Sigma_2$ theory, since if $\sigma$ is a $\Sigma_2$ sentence consistent with the $\Sigma_2$ theory of $M$, then $\sigma$ is $\sigma_n$ for some $n$, and we would have added $\sigma$ to the theory $T$ at stage $n$.
\end{proof}

\begin{theorem}
 For every model of set theory $M\satisfies\ZFC$, there is a $\Sigma_2$ assertion $\sigma(n)$ with some natural-number parameter $n\in\omega^M$, which is not true in $M$ but is consistent with the $\Sigma_2$ diagram of $M$.
\end{theorem}

\begin{proof}
Let $\sigma(n)$ be the assertion that the universal set $\set{a\mid\varphi(a)}$ defined in the main theorem has at least $n$ elements. This is a $\Sigma_2$ assertion, but for some $n$ it is not yet true, although it is consistent with the $\Sigma_2$ diagram of $M$, since it is true in a top extension of any countable elementary substructure of $M$.
\end{proof}

\begin{corollary}\label{Corollary.No-omega-standard-maximal-Sigma_2-theory}
 No $\omega$-standard model of \ZFC\ has a maximal $\Sigma_2$ theory. In particular, no transitive model of \ZFC\ has a maximal $\Sigma_2$ theory.
\end{corollary}

\begin{proof}
If $M$ is $\omega$-standard, then the parameter $n$ in the previous theorem is definable, and so $\sigma(n)$ can be taken as a $\Sigma_2$ sentence, not yet true in $M$ but consistent with the $\Sigma_2$ theory of $M$. So the $\Sigma_2$-theory of $M$ was not maximal.
\end{proof}

One may weaken the hypothesis of the corollary from $M$ being $\omega$-standard to assume only that every natural number of $M$ is $\Sigma_2$-definable in $M$. No such countable model can realize a maximal $\Sigma_2$ theory, since we can always move to a top-extension with a new $\Sigma_2$ assertion becoming true. In the previous arguments, we could have used the assertions, ``stage $n$ is successful,'' rather than the assertions about the size of the universal finite set.

\begin{theorem}
 In any countable model of set theory $M$, every element of $M$ becomes $\Sigma_2$ definable from a natural number parameter in some top-extension of $M$. Indeed, there is a single definition and single parameter $n\in\omega^M$, such that every $a\in M$ is defined by that definition with that parameter in some top-extension $M^+$.
\end{theorem}

\begin{proof}
 Let $n$ be the first unsuccessful stage of the universal finite set. For any object $a\in M$, it follows by the main theorem that there is a top extension $M^+$ of $M$ such that stage $n$ is now successful in $M^+$ and $a$ is the only set added at stage $n$. So $a$ has become $\Sigma_2$-definable in $M^+$ from parameter $n$, and the definition does not depend on $a$.
\end{proof}

\begin{corollary}
 For any countable $\omega$-standard model of set theory $M$, every object $a$ of $M$ becomes $\Sigma_2$ definable without parameters in some top-extension $M^+$ of $M$. Since $M^+$ is also $\omega$-standard, the result can be iterated.
\end{corollary}

\begin{proof}
This follows immediately from the preceding theorem, since the number $n$ used in that proof is standard finite and hence definable in a way that is absolute to further top-extensions.
\end{proof}

By iterating the previous corollary, we can form a top-extensional tower of models of \ZFC,
$$M_0\of M_1\of M_2\of\cdots$$
such that every object in any $M_n$ is $\Sigma_2$ definable without parameters in some later $M_k$. The union will not in general be a nice model, for the universal finite set $\set{x\mid \varphi(x)}$ as defined in the union will be internally infinite, with no last successful stage.

\section{Top-extensional set-theoretic potentialism}

Let us now introduce and consider the theory of top-extensional set-theoretic potentialism, concerned with the collection of countable models of \ZFC\ set theory, each accessing precisely its top-extensions. This is a potentialist system in the sense of~\cite{HamkinsLinnebo:Modal-logic-of-set-theoretic-potentialism}, where Hamkins and Linnebo considered several kinds of set-theoretic potentialism, determining in each case the corresponding potentialist validities and investigating the potentialist maximality principles. That investigation had built on earlier work in the modal logic of forcing, another potentialist concept in set theory; Hamkins introduced the forcing modality in~\cite{Hamkins2003:MaximalityPrinciple}; Hamkins and \Lowe\ determined the modal logic of forcing in~\cite{HamkinsLoewe2008:TheModalLogicOfForcing}, with further analysis in~\cite{HamkinsLoewe2013:MovingUpAndDownInTheGenericMultiverse} and~\cite{HamkinsLeibmanLoewe2015:StructuralConnectionsForcingClassAndItsModalLogic}.

Here, we extend the analysis to the case of top-extensional set-theoretic potentialism, a case not treated in~\cite{HamkinsLinnebo:Modal-logic-of-set-theoretic-potentialism}. Our analysis amounts to a set-theoretic analogue of~\cite{Hamkins:The-modal-logic-of-arithmetic-potentialism}, where Hamkins considered the models of \PA\ under top-extensions, analyzing the theory of top-extensional arithmetic potentialism.

\begin{figure}[h]
\begin{tikzpicture}[scale=.5,yscale=.9]
\node (0) at (0,0) {};
\node (a) at (-1,4) {};
\node (b) at (1,4) {};
\node (c) at (-4,7) {};
\node (d) at (-1,7) {};
\node (e) at (0,6) {};
\node (f) at (2.5,6) {};
\node (g) at (2,8.5) {};
\node (h) at (5,8.5) {};
\node (i) at (-2,9) {};
\node (j) at (1.5,9) {};
 \draw[fill=blue,fill opacity=.05,thin] (0.center) -- (a.center) -- node[below,opacity=1] {$M$} (b.center) -- cycle;
 \draw[fill=blue,fill opacity=.08] (0.center) -- (a.center) to[out=104,in=-60] (c.center) -- node[below,opacity=1,scale=.7] {$N_0$} (d.center) to[out=-85,in=76] (b.center) -- cycle;
 \draw[fill=red,fill opacity=.08,thin] (0.center) -- (a.center) to[out=104,in=-135] (e.center) -- node[below,opacity=1,scale=.7] {$N_1$} (f.center) to[out=-135,in=76] (b.center) -- cycle;
 \draw[fill=red,fill opacity=.08] (0.center) -- (a.center) to[out=104,in=-135] (e.center) to[out=45,in=-84] (g.center) -- node[below,opacity=1,scale=.7] {$N_{11}$} (h.center) to[out=-125,in=45] (f.center) to[out=-135,in=76] (b.center) -- cycle;
 \draw[fill=yellow,fill opacity=.08] (0.center) -- (a.center) to[out=104,in=-135] (e.center) to[out=45,in=-75] (i.center) -- node[below,opacity=1,scale=.7] {$N_{10}$} (j.center) to[out=-104,in=45] (f.center) to[out=-135,in=76] (b.center) -- cycle;
\end{tikzpicture}
\caption{A model of set theory $M$ with several top-extensions; they form a tree}
\end{figure}
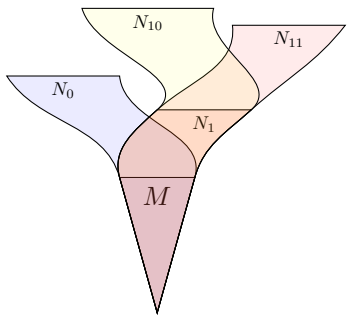

As a Kripke model, the set-theoretic top-extensional system provides natural interpretations of the modal operators. Namely, $\possible\varphi$ is true at a countable model of set theory $M$ if there is a top-extension of $M$ to a model $N$ in which $\varphi$ is true, and $\necessary\varphi$ is true at $M$ if all top-extensions $N$ of $M$ satisfy $\varphi$. One aims to study the modal validities of this system.

Specifically, an assertion $\varphi(p_0,\ldots,p_n)$ of propositional modal logic, expressed with propositional variables, Boolean connectives and modal operators, is \emph{valid} at a world with respect to a language of assertions, if $\varphi(\psi_0,\ldots,\psi_j)$ is true at that world for any assertions $\psi_i$ in that substitution language. In general, the question of whether a given modal assertion $\varphi$ is valid depends on the language of allowed substitution instances, such as whether parameters from $M$ are allowed or not. Following~\cite{HamkinsLinnebo:Modal-logic-of-set-theoretic-potentialism}, let us denote by $\Val(M,\mathcal{L}_\in(A))$ the set of propositional modal assertions $\varphi(p_0,\ldots,p_j)$ that are valid in $M$ with respect to assertions in the language of set theory allowing parameters from $A$. Allowing a larger substitution language or larger set of parameters is, of course, a more stringent requirement for validity, because of the additional substitution instances. So $\Val(M,\mathcal{L}_\in(A))$ is inversely monotone in $A$.

\begin{theorem}\label{Theorem.Validities-of-model-with-parameters-are-S4}
 Consider the potentialist system consisting of the countable models of \ZFC, each accessing its top extensions. For any model of set theory $M$, the modal assertions that are valid in $M$ for assertions in the language of set theory allowing parameters from $M$ or even just natural number parameters (and a single particular natural number $n\in\N^M$ suffices), are exactly the assertions of the modal theory \theoryf{S4}. Meanwhile, the validities with respect to sentences are contained within \theoryf{S5}. Succinctly,
\begin{equation}\notag\nonumber
\begin{split}
   S4 \ = \ \Val\left(M,\mathcal{L}_\in(M)\right) \
        &= \ \Val\left(M,\mathcal{L}_\in(\N^M)\right) \\
        &= \ \Val\left(M,\mathcal{L}_\in(\{n\})\right) \
        \of \ \Val\left(M,\mathcal{L}_\in\right) \ \of \ \theoryf{S5}.\\
\end{split}
\end{equation}
\end{theorem}

The inclusions stated at the end are optimal. Specifically, the particular parameter $n$ that suffices is the first unsuccessful stage of the universal finite set in $M$. When this is a standard number, then it is absolutely definable and hence not needed as a parameter, and in such a case the sentential validities are $\Val(M,\mathcal{L}_\in)=\theoryf{S4}$, realizing the lower bound. Meanwhile, theorem~\ref{Theorem.Maximality-principle} shows that some models have sentential validities $\Val(M,\mathcal{L}_\in)=\theoryf{S5}$, realizing the upper bound.

\begin{proof}
It is easy to see that every assertion of \theoryf{S4} is valid, regardless of the language for the substitution instances, for indeed, \theoryf{S4} is valid in any potentialist system. The content of the theorem is that, in contrast to many of the potentialist systems analyzed in~\cite{HamkinsLinnebo:Modal-logic-of-set-theoretic-potentialism}, there are no additional validities beyond \theoryf{S4} for this potentialist system, using the language of set theory with parameters.

For this, we follow the main technique of~\cite{Hamkins:The-modal-logic-of-arithmetic-potentialism}, where it was proved that the potentialist validities of the models of arithmetic under top-extension are also exactly \theoryf{S4}. Suppose that $\varphi(p_0,\ldots,p_j)$ is a propositional modal assertion not in \theoryf{S4}. Let $M$ be any countable model of \ZFC, and let $n$ be the first unsuccessful stage in the definition of the universal finite set in $M$. We will provide a substitution instance $\varphi(\psi_0(n),\ldots,\psi_j(n))$ of the formula $\varphi$ that is false at $M$ in the top-extensional potentialist semantics, using parameter $n$.

Since the collection of finite pre-trees is a complete set of frames for \theoryf{S4}, there is a Kripke model $K$ of propositional worlds, whose underlying frame is a finite pre-tree $T$, where $\varphi(p_0,\ldots,p_j)$ fails at an initial world.
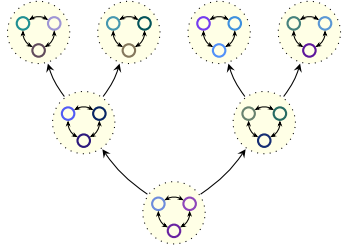
\begin{figure}[h]
\newcommand{\pentacluster}{\foreach\t in {0,...,4} {\draw (18+72*\t:1) node[circle,draw] (\t) {};}
    \foreach \r/\s in {0/1,1/2,2/3,3/4,4/0} {\draw[<->,>=stealth] (\r) edge[bend right=20] (\s);}
    \foreach \r/\s in {0/2,2/4,4/1,1/3,3/0} {\draw[<->,>=stealth] (\r) edge[bend right=10] (\s);}
    }
\newcommand{\tricluster}{\foreach\t in {0,1,2} {
  \setrandomcolor
    \draw (30+120*\t:.8) node[circle,thick,draw,scale=.5,randomcolor!65!blue] (\t) {};}
    \foreach \r/\s in {0/1,1/2,2/0} {\draw[{<[scale=.5]}-{>[scale=.5]},>=Stealth] (\r) edge[bend right=30] (\s);}
    }
\begin{tikzpicture}[scale=.3]
\begin{scope}
  \setrandomcolor
    \draw node[circle,dotted,draw,scale=2.5,fill=yellow,fill opacity=.1] (root) {};
    \tricluster
\end{scope}
\begin{scope}[shift={(-4,4)}]
  \setrandomcolor
    \draw node[circle,dotted,draw,scale=2.5,fill=yellow,fill opacity=.1] (left) {};
    \tricluster\end{scope}
\begin{scope}[shift={(4,4)}]
  \setrandomcolor
    \draw node[circle,dotted,draw,scale=2.5,fill=yellow,fill opacity=.1] (right) {};
    \tricluster\end{scope}
\begin{scope}[shift={(-6,8)}]
  \setrandomcolor
    \draw node[circle,dotted,draw,scale=2.5,fill=yellow,fill opacity=.1] (leftleft) {};
    \tricluster\end{scope}
\begin{scope}[shift={(-2,8)}]
  \setrandomcolor
    \draw node[circle,dotted,draw,scale=2.5,fill=yellow,fill opacity=.1] (leftright) {};
    \tricluster\end{scope}
\begin{scope}[shift={(2,8)}]
  \setrandomcolor
    \draw node[circle,dotted,draw,scale=2.5,fill=yellow,fill opacity=.1] (rightleft) {};
    \tricluster\end{scope}
\begin{scope}[shift={(6,8)}]
  \setrandomcolor
    \draw node[circle,dotted,draw,scale=2.5,fill=yellow,fill opacity=.1] (rightright) {};
    \tricluster\end{scope}
\draw[->,>=stealth] (root) edge[bend left=10] (left)
 (left) edge[bend left=10] (leftleft)
 (left) edge[bend right=10] (leftright)
 (root) edge[bend right=10] (right)
 (right) edge[bend left=10] (rightleft)
 (right) edge[bend right=10] (rightright);
\end{tikzpicture}
\caption{A pre-tree of possible worlds}\label{Figure.Pretree}
\end{figure}
Let $k$ be such that this tree is at most $k$-branching, and let $m$ be such that all the clusters have size at most $m$.

Consider the universal finite set as defined in $M$ and its top-extensions. This set is the result of a finite sequence of successful stages, each stage adding finitely many additional sets to the universal finite set. Thus, the process produces a definable finite sequence of natural numbers. From that list, but starting only from stage $n$, we may extract the subsequence consisting of the numbers less than $k$, which we may interpret as describing a particular way of successively climbing up the clusters of $T$, directing us to choose a particular branching cluster at each step, plus the last number on the list (from stage $n$ onward) that is $k$ or larger (or $0$ if there is none), which we may interpret as picking a particular node in that cluster by considering it modulo $m$. In this way, we associate with each node $t$ of the pre-tree $T$ a statement $\Phi_t$ that describes something about the nature of the sequence enumerated by the universal finite set process, in such a way that any world $N$ satisfying $\Phi_t$ will satisfy $\possible\Phi_r$ just in case $t\leq r$ in $T$. The reason is that the main theorem shows that the universal sequence as just defined can be extended in any desired finite way in a top extension. Thus, we have provided what is called a labeling for the finite pre-tree in~\cite{HamkinsLeibmanLoewe2015:StructuralConnectionsForcingClassAndItsModalLogic}; this is a \emph{railyard} labeling in the sense of~\cite{Hamkins:The-modal-logic-of-arithmetic-potentialism}. It follows from~\cite[lemma~9]{HamkinsLeibmanLoewe2015:StructuralConnectionsForcingClassAndItsModalLogic} (see also~\cite[theorem~28]{Hamkins:The-modal-logic-of-arithmetic-potentialism}) that there are sentences $\psi_0(n)$, $\psi_1(n)$,\ldots,$\psi_j(n)$, that track the truth of the propositional variables in the worlds of the propositional Kripke model $K$, using the parameter $n$, so that $N\satisfies\varphi(\psi_0(n),\ldots,\psi_j(n))$ just in case $(K,t)\satisfies\varphi(p_0,\ldots,p_j)$, where $t$ is the unique world of $K$ with $N\satisfies\Phi_t$. We may assume that having nothing at stage $n$ or beyond corresponds to the initial world of $K$, where $\varphi(p_0,\ldots,p_j)$ fails. And since the model of set theory $M$ has nothing at stage $n$ or beyond, it follows that $M\satisfies\neg\varphi(\psi_0(n),\ldots,\psi_j(n))$, and we have therefore found the desired substitution instance showing that $\varphi$ is not valid in $M$ for assertions with parameter $n$.

Lastly, when parameters are not allowed, it remains to show that the sentential validities are contained in \theoryf{S5}. For this, it suffices by the main results of~\cite{HamkinsLeibmanLoewe2015:StructuralConnectionsForcingClassAndItsModalLogic} to show that the model supports arbitrarily large families of independent switches, or alternatively a dial (a \emph{dial} in a Kripke model is a sequence of statements, such that every world satisfies exactly one of them and the possibility of all the others). For any standard finite number $k$, let $d_k$ be the sentence asserting that the last successful stage of process $C$ in the main theorem admits a set of size $k$ to the universal finite set. These sentences can be used to form a dial of any desired finite size, since we can always top-extend so as to make any one of them true. So the modal validities of any model of set theory $M$ in the system will be contained within \theoryf{S5}.
\end{proof}

\begin{corollary}
 If there is a transitive model of \ZFC, or indeed merely a model of \ZFC\ in which the universal finite set is empty or has standard finite size, then the modal assertions that are valid in every countable model of set theory, with respect to sentences of the language of set theory, are exactly the assertions of \theoryf{S4}.
\end{corollary}

\begin{proof}
 Suppose that $M$ is a model of \ZFC\ set theory in which the universal finite set has standard finite size. So the first unsuccessful stage $n$ in $M$ is a standard finite number. Thus, we do not need it as a parameter in theorem~\ref{Theorem.Validities-of-model-with-parameters-are-S4}, since it is definable in a way that is absolute to top-extensions. So the top-extensional validities of $M$, with respect to sentences in the language of set theory, are exactly \theoryf{S4}.
\end{proof}

Meanwhile, there are other models, whose sentential validities strictly exceed \theoryf{S4}, showing that the validities can depend on whether one allows parameters or not. In order to prove this, let us first establish the following characterization of top-extensional possibility. This lemma is the set-theoretic analogue of the corresponding analysis of~\cite{Hamkins:The-modal-logic-of-arithmetic-potentialism} for models of arithmetic.

\goodbreak
\begin{lemma}[Possibility-characterization lemma]\label{Lemma.Possibility-characterization}
 In the top-extensional potentialist system consisting of the countable models of $\ZFC$ under top-extensions, the following are equivalent for any countable $\omega$-nonstandard model of set theory $M$ and any assertion $\varphi(a)$ in the language of set theory with parameter $a\in M$.
 \begin{enumerate}
   \item $M\satisfies\possible\varphi(a)$. That is, $\varphi(a)$ is true in some top-extension of $M$.
   \item For every standard number $k$ and every ordinal $\beta$ in $M$, in the collapse extension $M[G]$ where $V_\beta^M$ is made countable, there is a top-extension of $\langle V_\beta^M,\in\rangle$ to a model $\<N,\in^N>$ satisfying $\ZFC_k$ and $\varphi(a)$.
   \item There is some nonstandard $k$, such that for every ordinal $\beta$ in $M$ there is in the collapse extension $M[G]$ a top-extension of $\langle V_\beta^M,\in\rangle$ to a model $\<N,\in>$ satisfying $\ZFC_k$ and $\varphi(a)$.
 \end{enumerate}
\end{lemma}

\begin{proof}
($2\to 3$) Overspill.

($3\to 1$) Fix the nonstandard number $k$ as in statement $3$. Let $M^+$ be a countable elementary top-extension of $M$, and let $\beta$ be some new ordinal of $M^+$ larger than the ordinals of $M$. In the collapse extension $M^+[G]$ in which $V_\beta^N$ is countable, there is a top-extension of $\langle V_\beta^{M^+},\in\rangle$ to a model $\<N,\in^N>$ in which $\ZFC_k$ and $\varphi(a)$ hold. By the choice of $\beta$ above $M$, and since $k$ is nonstandard, it follows that $N$ is a top-extension of $M$ to a model of \ZFC\ in which $\varphi(a)$, so $M\satisfies\possible\varphi(a)$, as desired.

($1\to 2$) Suppose that $M$ has a top-extension to a model $N$ of $\ZFC$ in which $\varphi(a)$ holds. For any standard finite $k$, by the reflection theorem in $N$, there are arbitrarily large ordinals $\delta$ for which $V_\delta^N\satisfies\ZFC_k\wedge\varphi(a)$. So for every ordinal $\beta$ in $M$, there is in $N$ a top-extension of $\langle V_\beta^M,\in\rangle$ to a model of $\ZFC_k$ plus $\varphi(a)$. Let $G$ be $N$-generic for the collapse forcing to make $V_\beta^M$ countable. So in $N[G]$ the $\Sigma^1_1$ statement asserting that there is such a model extending $\langle V_\beta^M,\in\rangle$ is true. By Shoenfield absoluteness, this statement is also true in $M[G]$, as this model has the same countable ordinals. So statement $2$ holds.
\end{proof}

It follows that $\necessary\varphi(a)$ is true in a countable $\omega$-nonstandard model of set theory $M$ just in case there is some ordinal $\beta$ and some standard number $k$ such that $M$ thinks that in the collapse forcing extension $M[G]$ making $V_\beta^M$ countable, there is no top-extension of $\langle V_\beta^M,\in\rangle$ to a model $\<N,\in>$ of $\ZFC_k$ in which $\neg\varphi(a)$ is true. In particular, every instance of necessity, for an assertion of any complexity, is true because a certain $\Sigma_2$ fact is true in $M$, namely, the existence of a $\beta$ for which the collapse forcing forces a certain fact about a certain standard natural number $k$.

Following the ideas of~\cite{HamkinsLinnebo:Modal-logic-of-set-theoretic-potentialism}, let us say that a countable model of set theory $M$ satisfies the \emph{top-extensional maximality principle}, if for any sentence $\varphi$ in the language of set theory, if there is a top extension $N$ of $M$ in which $\varphi$ is true and remains true in all further top-extensions of $N$, then $\varphi$ was already true in $M$. In modal terms, the top-extensional maximality principle is expressed by the scheme of assertions $\possible\necessary\varphi\to\varphi$, where the modal operators refer to top-extensional possibility and top-extensional necessity, respectively.

\begin{theorem}\label{Theorem.Maximality-principle}
  If there is a model of \ZFC, then there is a model of \ZFC\ satisfying the top-extensional maximality principle. Indeed, any countable model of \ZFC\ satisfying a maximal $\Sigma_2$ theory will satisfy the top-extensional maximality principle.
\end{theorem}

\begin{proof}
If there is any model of \ZFC, then observation~\ref{Observation.Maximal-Sigma_2-theory} shows that there is a model of set theory $M$ with a maximal $\Sigma_2$ theory. Fix any such model $M$. It must be $\omega$-nonstandard, since otherwise there are new $\Sigma_2$ sentences that could become true in a top extension, concerning the size or number of successful stages in the universal finite set. To verify the maximality principle in $M$, suppose that $M\satisfies\possible\necessary\sigma$, meaning that there is some top-extension of $M$ to a model $M^+$ in which $\sigma$ holds and continues to hold in all further top-extensions of $M^+$. Since $M^+\satisfies\necessary\sigma$, this means that $\neg\sigma$ is not possible over $M^+$, and so by lemma~\ref{Lemma.Possibility-characterization} there is some ordinal $\beta$ and standard number $k$ such that in the collapse extension of $M^+$ making $V_\beta^{M^+}$ countable, there is no top-extension of $\langle V_\beta^{M^+},\in^{M^+}\rangle$ to a model $\langle N,\in^N\rangle$ satisfying $\ZFC_k$ plus $\neg\sigma$. The existence of such an ordinal $\beta$ is a $\Sigma_2$ assertion that is true in $M^+$ and therefore consistent with the $\Sigma_2$ theory of $M$. By the maximality of the $\Sigma_2$ theory of $M$, it follows that there is such an ordinal $\beta$ already in $M$. Therefore, by lemma~\ref{Lemma.Possibility-characterization} again it follows that $M\satisfies\necessary\sigma$ and in particular, $\sigma$ is true in $M$. So $M$ satisfies the top-extensional maximality principle.
\end{proof}

The proof really shows that if a countable model of set theory $M$ satisfies the maximality principle for $\Sigma_2$ sentences, then it satisfies the maximality for all sentences in the language of set theory. The main theorem shows that no countable model of set theory $M$ satisfies the top-extensional maximality principle with respect to assertions allowing parameters from $M$, even merely natural number parameters.

\begin{corollary}
 If \ZFC\ is consistent, then there are models of \ZFC\ whose top-extensional potentialist validities, with respect to sentences in the language of set theory, are exactly \theoryf{S5}. $$\Val(M,\mathcal{L}_\in)=\theoryf{S5}.$$
\end{corollary}

\begin{proof}
The maximality principle amounts to \theoryf{S5} as a lower bound, and we already have \theoryf{S5} as an upper bound by theorem~\ref{Theorem.Validities-of-model-with-parameters-are-S4}.
\end{proof}

\printbibliography

\end{document}